\documentclass[11pt,eqno]{article}
\usepackage{amssymb,amsmath,latexsym}

\usepackage[dvips]{color}
\allowdisplaybreaks

\newtheorem{thm}{Theorem}[section]
\newtheorem{proposition}[thm]{Proposition}
\newtheorem{remark}[thm]{Remark}
\newtheorem{definition}[thm]{Definition}
\newtheorem{corollary}[thm]{Corollary}

\newtheorem{lemma}[thm]{Lemma}
\newtheorem{assumption}[thm]{Assumption}

\newcommand{\qed}{\hfill\rule{2mm}{3mm}\vspace{4mm}}

\def\cB{\mbox{${\cal B}$}}

\def\cF{\mbox{${\cal F}$}}
\def\cH{\mbox{${\cal H}$}}
\def\cL{\mbox{${\cal L}$}}
\def\cG{\mbox{${\cal G}$}}

\def\cT{\mbox{${\cal T}$}}

\def\proof{\noindent {\bf Proof. }\ }

\def\wh{\widehat}

\def\qed{\hfill $\square$ \bigskip}

\def\beq{\begin{equation}}               
\def\eeq{\end{equation}}                 
\def\bea{\begin{eqnarray}}             
\def\eea{\end{eqnarray}}               
\def\be*{\begin{eqnarray*}}             
\def\ee*{\end{eqnarray*}}               
\def\ba{\begin{array}}                  
\def\ea{\end{array}}                    

\def\;{\vspace{3mm} \\ }
\def\Q{\rm Q}
\def\N{\mathbb{N} }
\def\H{\mathbb{H} }

\def\K{\bf\rm{K} }
\def\R{\mathbb{R}}
\def\E{\mathbb{E} }
\def\P{\mathbb{P} }

\def\LL{\mathbb{L} }

\def\~{\widetilde}

\def\F{{\cal F}}
\def\eps{\varepsilon}

\def\beqlb{\begin{eqnarray}} \def\eeqlb{\end{eqnarray}}
\def\beqnn{\begin{eqnarray*}} \def\eeqnn{\end{eqnarray*}}

\def\<{\langle}  \def\>{\rangle}

\def\wh{\widehat}

\def\A{{\cal A}}  \def\L{{\cal L}}

\def\bde{\begin{definition}}
\def\ede{\end{definition}}

\def\bth{\begin{thm}}
\def\eth{\end{thm}}
\def\bpr{\begin{proposition}}
\def\epr{\end{proposition}}
\def\ble{\begin{lemma}}
\def\ele{\end{lemma}}
\def\bcor{\begin{corollary}}
\def\ecor{\end{corollary}}
\def\bre{\begin{remark}}
\def\ere{\end{remark}}

\begin{document}

\def\ee{\varepsilon}
\def\qed{{\hfill $\Box$ \bigskip}}
\def\MM{{\cal M}}
\def\BB{{\cal B}}
\def\LL{{\cal L}}
\def\FF{{\cal F}}
\def\EE{{\cal E}}
\def\QQ{{\cal Q}}
\def\AA{{\cal A}}

\def\R{{\bf R}}
\def\N{{\mathbb N}}
\def\L{{\mathbb L}}
\def\E{{\bf E}}
\def\F{{\bf F}}
\def\H{{\bf H}}
\def\P{{\bf P}}
\def\Q{{\bf Q}}
\def\S{{\bf S}}
\def\J{{\bf J}}
\def\K{{\bf K}}
\def\F{{\bf F}}
\def\A{{\bf A}}
\def\loc{{\bf loc}}
\def\eps{\varepsilon}
\def\semi{{\bf semi}}
\def\wh{\widehat}
\def\pf{\noindent{\bf Proof.} }
\def\dim{{\rm dim}}

\title{\Large \bf $L\log L$ condition for  supercritical branching\\  Hunt processes }
\author{ \bf Rong-Li Liu, \hspace{1mm}\hspace{1mm} Yan-Xia
Ren\footnote{The research of this author is supported by NSFC
(Grant No.  10871103 and 10971003)\hspace{1mm} } \hspace{1mm}\hspace{1mm} and
Renming Song\hspace{1mm} }

\date{}
\maketitle

{\narrower{\narrower

\centerline{\bf Abstract}

\bigskip

In this paper we use the spine decomposition and
martingale change of measure to establish a Kesten-Stigum $L\log
L$ theorem for branching Hunt processes. This result is a
generalization of the results in \cite{AH} and \cite{H} for
branching diffusions.

\bigskip


\medskip


\par}\par}

\bigskip

\begin{doublespace}

\section{Introduction}
  \setcounter{equation}{0}

Suppose that $\{Z_n; n\ge 1\}$ is a Galton-Watson process with each
particle having probability $p_n$ of giving birth to $n$ offspring.
Let $L$ stand for a random variable with this offspring
distribution. Let $m:=\sum^{\infty}_{n=1}np_n$ be the mean number of
offspring per particle. Then $Z_n/m^n$ is a non-negative martingale.
Let $W$ be the limit of $Z_n/m^n$ as $n\to\infty$. Kesten and Stigum
proved in \cite{KS} that if $1<m<\infty$ ( that is, in the
supercritical case) then $W$ is non-degenerate (i.\/e., not almost
surely zero) if and only if
\begin{equation}\label{LLogL-GW}
E(L\log ^+L))=\sum^{\infty}_{n=1}p_nn\log n<\infty,
\end{equation}
here, and in the rest of this paper, we use the notation that
$\log^+r=0\vee\log r$ for all $r> 0$.
This result is usually referred to  the Kesten-Stigum $L\log L$
theorem.

In 1995, Lyons, Pemantle and Peres developed a martingale change of
measure method in \cite{LPP} to give a new proof for the
Kesten-Stigum $L\log L$ theorem for single type branching processes.
Later this method was extended to prove the $L\log L$ theorem for
multiple and general multiple type branching processes in \cite{BK},
\cite{KLPP} and \cite{L}.

In this paper we will extend this method to supercritical branching
Hunt processes and establish an $L\log L$ criterion for branching
Hunt processes. To review the known results and  state our main
result, we need to introduce the setup we are going to work with
first.

In this paper $E$ always stands for a locally compact separable
metric space. We will use $E_{\Delta}:=E\cup\{\Delta \}$ to denote
the one-point compactification of $E$. We will use ${\cal B}(E)$ and
${\cal B}(E_{\Delta})$ to denote the Borel $\sigma$-fields on $E$
and $E_{\Delta}$, respectively. $\cB_b(E)$( respectively,
$\cB^+(E)$) will denote the set of all bounded (respectively,
non-negative) $\cB(E)$-measurable functions on $E$. All functions
$f$ on $E$ will be automatically extended to $E_{\Delta}$ by setting
$f(\Delta)=0$. Let ${\bf M}_p(E)$ be the space of finite point
measures on $E$, that is, measures of the form $\mu=\sum^n_{i=1}
\delta_{x_i}$ where $n=0, 1, 2, \dots$ and $x_i\in E, i=1, \dots,
n$. (When $n=0$, $\mu$ is the trivial zero measure.) For any
function $f$ on $E$ and any measure $\mu \in {\bf M}_p(E)$ we use
$\<f, \mu\>$ to denote the integral of $f$ with respect to $\mu$.

We will always assume that $Y=\{Y_t, \Pi_x, \zeta\}$ is a Hunt process
on $E$, where $\zeta=\inf\{t>0:\,  Y_t=\Delta\}$ is the lifetime of
$Y$. Let $\{P_t, t\ge 0\}$ be the transition semigroup of $Y$:
$$
P_tf(x)=\Pi_x[f(Y_t)]\quad \mbox{ for }f\in{\cal B}^+(E).
$$
Let $m$ be a positive Radon measure on $E$ with full support.
$\{P_t, t\ge 0\}$ can be extended to a strongly continuous semigroup
on $L^2(E, m)$. Let $\{\hat P_t, t\ge 0\}$ be the semigroup on
$L^2(E, m)$ such that
$$
\int_Ef(x)P_tg(x)m(dx)=\int_Eg(x)\hat P_tf(x)m(dx),\quad f,g\in
L^2(E,m).
$$
We will use ${\bf A}$ and $\hat{\bf A}$ to denote the generators of
the semigroups $\{P_t\}$ and $\{\widehat P_t\}$ on $L^2(E, m)$
respectively.

Throughout this paper we assume that

\begin{assumption}\label{assume1}
(i) There exists a family of continuous strictly positive functions
$\{p(t,\cdot,\cdot); t>0\}$ on $E\times E$ such that for any
$(t,x)\in (0,\infty)\times E$, we have
$$
P_tf(x)=\int_Ep(t, x, y)f(y)m(dy),\quad \hat P_tf(x)=\int_Ep(t, y,
x)f(y)m(dy).
$$
(ii) The semigroups $\{P_t\}$ and $\{\hat P_t\}$ are
ultracontractive, that is, for any $t>0$, there exists a constant
$c_t>0$ such that
$$p(t, x, y)\le c_t\quad\mbox{ for any } (x,y)\in E\times E.$$
\end{assumption}

Consider a branching system determined by the following three
parameters:
\begin{description}
\item{(a)} a Hunt process $Y=\{Y_t, \Pi_{x},\zeta\}$
with state space  $E$;
\item{(b)} a nonnegative bounded measurable function $\beta$ on $E$;
\item{(c)} an offspring distribution $\{(p_n(x))_{n=0}^\infty;\, x\in E\}$.
\end{description}
Put \beq \psi(x,z)=\sum^\infty_{n=0}p_n(x)z^n,\quad z\geq 0. \eeq
$\psi$ is the generating function for the number of offspring
generated at point $x$.

This branching system is characterized by the following properties:
\begin{description}
\item{(i)} Each particle has a random birth and a random death
time. \item{(ii)} Given that a particle is born at $x\in E$, the
conditional distribution of its path after birth is determined by
$\Pi_{x}$. \item{(iii)} Given the path $Y$ of a particle and given
that the particle is alive at time $t$, its probability of dying in
the interval $[t, t + dt)$ is $\beta(Y_t)dt + o(dt)$.
\item{(iv)} When a particle dies at $x\in E$, it
splits into $n$ particles at $x$ with probability $p_n(x)$.
\item{(v)} The point $\Delta$ is a cemetery. When a particle reaches
$\Delta$, it stays at $\Delta$ for ever and there is no branching at
$\Delta$.
\end{description}

We assume that the functions $p_n(x)$, $n=0,1,\cdots$, and
$A(x):=\psi'(x,1)=\sum^\infty_{n=0}np_n(x)$ are bounded
$\cB(E)$-measurable and that $p_0(x)+p_1(x)=0$ on  $E$.
The last condition implies $A(x)\geq 2$ on $E$. The assumption
$p_0(x)=0$ on $E$ is essential for the probabilistic proof of this paper
since we need the spine to be defined for all $t\ge 0$. The assumption
$p_1(x)=0$ on $E$ is just for convenience as the case $p_1(x)>0$
can be reduced to the case $p_1(x)=0$ by changing the parameters $\beta$
and $\psi$ of the branching Hunt process.

For any $c\in\cB_b(E)$, we define
$$
e_c(t)=\exp\left(-\int^t_0c(Y_s)ds\right).
$$
Let $X_t(B)$ be the number of particles located in $B\in\cB(E)$ at
time $t$. Then $\{X_t, t\ge 0\}$ is a Markov process in ${\bf
M}_p(E)$. This process is called a $(Y, \beta, \psi)$-branching Hunt
process. For any $\mu\in{\bf M}_p(E)$, let $P_{\mu}$ be the law of
$\{X_t, t\ge 0\}$ when $X_0=\mu$. Then we have
 \beq
P_{\mu}\exp\<-f, X_t\>=\exp\<\log u_t(\cdot),\mu\>
 \eeq
where $u_t(x)$ satisfies the equation
 \beq\label{int}
u_t(x)=\Pi_{x}\left[e_{\beta}(t)\exp(-f( Y_t))+\int^t_0
e_{\beta}(s)\beta(Y_s)\psi( Y_s, u_{t-s}(
Y_s))ds\right]\quad\mbox{for }  t\ge 0.
 \eeq

The formula \eqref{int} deals with a process started at time $0$
with one particle located at $x$, and it has a clear heuristic
meaning: the first term in the brackets corresponds to the case when
the particle is still alive at time $t$; the second term corresponds
to the case when it dies before $t$. The formula \eqref{int} implies
that
 \beq\label{int2}
u_t(x)=\Pi_{x}\int^t_0\left[\psi(Y_s,u_{t-s}(Y_s))-u_{t-s}(Y_s)\right]\beta(Y_s)ds+
\Pi_{x}\exp(-f(Y_t))\quad\mbox{for } t\ge 0
 \eeq
(see \cite[Section 2.3]{Dy}).
For any $\mu\in{\bf M}_p(E), f\in {\cal B}_b^+ (E)$ and $t\geq 0$,
we have
\begin{eqnarray}\label{expX}
P_{\mu} \left[\langle  f, X_t\rangle\right]
=\Pi_{\mu}\left[e_{(1-A)\beta}(t)f(Y_t)\right].
\end{eqnarray}

Let $\{P^{(1-A)\beta}_t,t\ge 0\}$ be the Feynman-Kac semigroup
defined by
$$
P^{(1-A)\beta}_tf(x):=\Pi_x\left[f(Y_t)\
e_{(1-A)\beta}(t)f(Y_t)\right], \qquad
f\in{\cal B}(E).
$$
Let $\{\widehat P^{(1-A)\beta}_t,t\ge 0\}$ be the dual semigroup of
$\{P^{(1-A)\beta}_t,t\ge 0\}$ on $L^2(E, m)$.

Under Assumption \ref{assume1}, we can easily show that the
semigroups $\{P^{(1-A)\beta}_t\}$ and $\{\widehat
P^{(1-A)\beta}_t\}$ are strongly continuous on $L^2(E, m)$. Moreover,
$P^{(1-A)\beta}_t$ admits a density $p^{(1-A)\beta}(t,x,y)$ that is
jointly continuous in $(x, y)$ for each $t>0$:
$$
P^{(1-A)\beta}_tf(x) =\int_E p^{(1-A)\beta}(t,x,y)f(y)m(dy),\quad\mbox{ for
every } f\in\cB^+(E).
$$
The generators of $\{P^{(1-A)\beta}_t\}$ and $\{\widehat
P^{(1-A)\beta}_t\}$ can be formally written as $\A+(A-1)\beta$ and
$\widehat{\A}+(A-1)\beta$ respectively.

Let $\sigma(\A+(A-1)\beta)$ and $\sigma(\widehat{\A}+(A-1)\beta)$
denote the spectrum of operator $\A+(A-1)\beta$ and
$\widehat{\A}+(A-1)\beta$, respectively. It follows from Jentzch's
Theorem (Theorem V.6.6 on page 333 of \cite{Sc} ) and the strong
continuity of $\{P^{(1-A)\beta}_t\}$ and $\{\widehat
P^{(1-A)\beta}_t\}$ that the common value $\lambda_1:= \sup {\rm
Re}(\sigma(\A+(A-1)\beta))= \sup {\rm
Re}(\sigma(\widehat{\A}+(A-1)\beta))$ is an eigenvalue of
multiplicity 1 for both $\A+(A-1)\beta$ and
$\widehat{A}+(A-1)\beta$, and that an eigenfunction $\phi$ of
$\A+(A-1)\beta$ associated with $\lambda_1$ can be chosen to be
strictly positive a.e. on $E$ and an eigenfunction
$\widetilde{\phi}$ of $\widehat{\A}+(A-1)\beta$ associated with
$\lambda_1$ can be chosen to be strictly positive a.e. on $E$.
By \cite[Proposition 2.3]{KS1} we know that $\phi$ and
$\widetilde{\phi}$ are strictly positive and continuous on $E$. We
choose $\phi$ and $\widetilde{\phi}$ so that
$\int_{E}\phi(x)\widetilde{\phi}(x)m(dx)=1$. Then
\begin{equation}\label{invar}
\phi(x)=e^{-\lambda_1t}P^{(1-A)\beta}_t\phi(x),\quad
\widetilde\phi(x)=e^{-\lambda_1t}\widehat
P^{(1-A)\beta}_t\widetilde{\phi},\quad x\in E.
\end{equation}

Throughout this paper we assume that

\begin{assumption}\label{assume2}$\lambda_1>0$.\end{assumption}

The above assumption is the condition for the supercriticality of the branching
Hunt process. Indeed, if $\lambda_1< 0$, it is easy to see that extinction
occurs almost surely from the martingale $M_t(\phi)$ defined below.

Let ${\cal E}_t=\sigma(Y_s;\ s\leq t)$. Note that
$$
\frac{\phi(Y_{t})}{\phi(x)}e^{-\lambda_1t}e_{(1-A)\beta}(t)
$$
is a martingale under $\Pi_x$, and so we can define a martingale
change of measure by
$$
\frac{d\Pi_x^\phi}{d\Pi_x}\Big|_{\mathcal{E}_t}= \frac{\phi(Y_{t})}
{\phi(x)}e^{-\lambda_1t}e_{(1-A)\beta}(t).
$$
Then $\{Y,\ \Pi_x^\phi\}$ is a conservative Markov process,
and $\phi\widetilde{\phi}$ is
the unique invariant probability density for the semigroup $P^{(1-A)\beta}_t$,
that is, for any $f\in \cB^+(E)$,
$$
\int_E\phi(x)\widetilde{\phi}(x)P^{(1-A)\beta}_tf(x)m(dx)=\int_Ef(x)\phi(x)\widetilde{\phi}(x)m(dx).
$$
Let $p^\phi(t, x,y)$ be the transition density of $Y$ in $E$ under
$\Pi^\phi_x$. Then
$$
p^\phi(t,x,y)=\frac{e^{-\lambda_1t}}{\phi(x)}\ p^{(1-A)\beta}(t, x,
y)\ \phi(y).
$$

Throughout this paper, we assume the following

\begin{assumption}\label{assume3}
The semigroups $\{P^{(1-A)\beta}_t\}$ and $\{\widehat
P^{(1-A)\beta}_t\}$ are  intrinsic ultracontractive , that is, for
any $t > 0$ there exists a constant $c_t$ such that
$$
p^{(1-A)\beta}(t, x, y)\le c_t\phi(x)\widetilde{\phi}(y), \quad x, y
\in E.
$$
\end{assumption}

\begin{remark} Here are some examples of Hunt processes
satisfying Assumptions \ref{assume1} and \ref{assume3}.

(1) Suppose $E=D$, a domain in $\mathbb{R}^d$, and $m$ is the
Lebesgue measure on $D$. If $\{Y,\ \Pi_x,\ x\in D\}$ is a diffusion
killed upon leaving $D$ with generator
$$
\A=\frac{1}{2}\nabla\cdot a\nabla+b\cdot \nabla
$$
where $(a_{ij}(x))_{ij}$ is uniformly elliptic and  bounded with
$a_{ij}, i,j=1,\cdots, d$, being bounded functions in
$C^1(\mathbb{R}^d)$ such that all their first partial derivatives
are bounded, and  $b_i, i=1, \cdots, d$,  are bounded Borel
functions on $\mathbb{R}^d$. It was proven in \cite{KS1} and
\cite{KS2} that the semigroups $\{P^{(1-A)\beta}_t\}_{t\ge 0}$ and
$\{\widehat P^{(1-A)\beta}_t\}$ are intrinsic ultracontractive when
$D$ is a bounded Lipschitz domain. For more general conditions on
$D$ and the coefficients for $\{P^{(1-A)\beta}_t\}_{t\ge 0}$ and
$\{\widehat P^{(1-A)\beta}_t\}$ to be intrinsic ultracontractive,
see \cite{KS2}.

(2) Suppose $E=D$, a bounded open set in $\mathbb{R}^d$, and $m$ is
the Lebesgue measure on $D$. If $\{Y,\ \Pi_x,\ x\in D\}$ is a
symmetric $\alpha-$stable process killed upon exiting $D$, where
$0<\alpha<2$,  then it follows from \cite{CS1} and \cite{Ku1} that
the semigroups $\{P^{(1-A)\beta}_t\}_{t\ge 0}$ and $\{\widehat
P^{(1-A)\beta}_t\}_{t\ge 0}$ are intrinsic ultracontractive.

(3) For examples in which $E$ is unbounded, see \cite{KB} and \cite{Kw}.

(4) For more examples of discontinuous Markov processes satisfying
Assumptions \ref{assume1} and \ref{assume3}, we refer our readers to
\cite{KS3} and the references therein.
\end{remark}

It follows from \cite[Theorem 2.8]{KS1} that
\begin{equation}
\left|\frac{e^{-\lambda_1t}p^{(1-A)\beta}(t, x,y)} {\phi(x)
\tilde\phi(y)}-1\right|\le c\,e^{-\nu t},\quad x\in E,
\end{equation}
for some positive constants $c$ and $\nu$, which is equivalent to
\begin{equation}\label{IU}
\sup_{x\in E}\left|\frac{p^\phi(t, x,y)}{\phi(y) \tilde\phi(y)}-
1\right|\le c\,e^{-\nu t}.
\end{equation}
Thus for any $f\in\cB^+_b(E)$ we have
$$
\sup_{x\in E}\left|\int_E p^\phi(t, x, y)f(y)m(dy) - \int_E\phi(y)
\tilde\phi(y)f(y)m(dy)\right|\le c\,e^{-\nu t}\int_E\phi(y)
\tilde\phi(y)f(y)m(dy).
$$
Consequently we have
\begin{equation}\label{u-convergent'}
\lim_{t\to\infty}\displaystyle\frac{\int_E p^\phi(t, x, y)
f(y)m(dy)}{\int_E\phi(y) \tilde\phi(y)f(y)m(dy)}=1,\quad \mbox{
uniformly for } f\in\cB^+_b(E)\mbox{ and }x\in E.
\end{equation}

For any nonzero measure $\mu\in {\bf M}_p(E)$, we define
$$
M_t(\phi):=e^{-\lambda_1 t}\frac{\langle\phi,
X_t\rangle}{\langle\phi, \mu\rangle}\qquad t\geq 0.
$$

\begin{lemma}\label{l:1.5}  For any nonzero measure $\mu\in {\bf M}_p(E)$,
$M_t(\phi)$ is a nonnegative martingale under $P_{\mu}$, and
therefore there exists a limit $ M_{\infty}(\phi)\in[0,\infty)$,
$P_{\mu}$-a.s.
\end{lemma}

\proof By the Markov property of $X$, \eqref{expX}  and
\eqref{invar},
\begin{eqnarray*}
P_{\mu} \left[M_{t+s}(\phi) \big| \cF_t\right] &=& \frac{1}{\<\phi,
\mu\>}e^{-\lambda_1 t} P_{X_t} \left[e^{-\lambda_1 s} \< \phi, X_s\>\right]\\
& =&\frac{1}{\<\phi, \mu\>}e^{-\lambda_1 t} \left\<  e^{-\lambda_1
s}  \Pi_{\cdot}\left[ e_{(1-A)\beta}(s)\phi(Y_s)\right], \, X_t\right\> \\
&=&\frac{1}{\<\phi, \mu\>}e^{-\lambda_1 t} \< \phi, \, X_t\> =
M_t(\phi).
\end{eqnarray*}
This proves that $\{M_t(\phi), t\ge 0\}$ is a non-negative
$P_\mu$-martingale and so it has an almost sure limit
$M_\infty(\phi)\in[0,\infty)$ as $t\to \infty$. \qed

In this paper we are concerned with the following classical
question: under what condition is the limit $M_\infty(\phi)$
non-degenerate (that is, $P_{\mu}(M_\infty(\phi)>0)>0$)? In
\cite{AH}, Asmussen and Hering  gave a criterion for
$M_{\infty}(\phi)$ to be non-degenerate for a general class of
branching Markov processes under regularity conditions. More
precisely, it was proved in \cite{AH} that if the underlying Markov
process $Y$ is positive regular (see \cite{AH} for the precise
definition), $M_{\infty}(\phi)$ is non-degenerate if and only if
 \beq\label{LlogL-t}
\int_E m(dy)\widetilde{\phi}(y)P_{\delta_y}\left[\<\phi,
X_t\>\log^+\<\phi, X_t\>\right]<\infty\quad \mbox{ for some }t>0.
 \eeq
This condition is not easy to verify since it involves the branching
process $X$ itself. It is more desirable to have a criterion in
terms of the natural model parameters $\A$, $\beta$ and $\{p_n(x)\}$
of the branching process. Such a criterion is found in \cite{AH} and
\cite{H} for branching diffusions and it was proved that, in the
case of branching diffusions on a bounded open set
$E\subset\mathbb{R}^d$ with $E$ being the union of  finite number of
bounded $C^3$-domains, $M_{\infty}(\phi)$ is non-degenerate if and
only if
\begin{equation}
\int_E\widetilde{\phi}(y)\beta(y)l(y)m(dy)<\infty.
\end{equation}
where
 \beq\label{def-l}l(x)=\sum_{k=2}^\infty k\phi(x)\log^+( k\phi(x))\,
p_k(x),\quad x\in E.
 \eeq
The arguments of \cite{AH} and \cite{H} are mainly analytic.

The purpose of this paper is two-folds. First, we generalize the
result above to general branching Hunt processes. Even in the case
of branching diffusions, our main result is more general than the
corresponding result in \cite{H} in some aspect since our
requirement on the regularity of the domain is very weak. Secondly,
we give a more probabilistic  proof of the result, using the spine
decomposition and martingale change of measure.
Our probabilistic proof is similar to the probabilistic proofs of
\cite{HH1}, \cite{K2} and \cite{LPP}.

The main result of this paper can be stated as follows.

\begin{thm}\label{maintheorem} Suppose that
$\{X_t; t\ge 0\}$ is a $(Y,\beta, \psi)$-branching Hunt process and
that Assumptions 1.1--1.3 are satisfied.
Then $M_\infty(\phi)$ is non-degenerate under $P_\mu$ for any
nonzero measure $\mu\in {\bf M}_p(E)$ if and only if
 \beq\label{LlogL-BH}
\int_E\widetilde{\phi}(x)\beta(x)l(x)m(dx)<\infty,
 \eeq
where $l$ is defined in \eqref{def-l}.
\end{thm}

It follows from the branching property that when $\mu\in {\bf
M}_p(E)$ is given by $\mu=\sum_{i=1}^n\delta_{x_i}, n=1, 2, \dots,
\{x_i; i=1,\cdots, n\}\subset E$, we have
$$
M_t(\phi)=\sum_{i=1}^n e^{-\lambda_1 t}\frac{\langle \phi^t,
X_t^i\rangle}{\phi(x_i)} \cdot\frac{\phi(x_i)}{\langle\phi,
\mu\rangle},
$$
where $X^i_t$ is a branching Hunt process starting from
$\delta_{x_i}, i=1, \dots, n$. If the conclusions hold for the cases
that $\mu=\delta_x, $ for any $x\in E$, then the conclusions also
hold for the general cases.  So in the remainder of this paper, we assume that the
initial measure is of the form $\mu=\delta_x, x\in E,$ and
$P_{\delta_x}$ will be denoted as $P^x$.

This paper is organized as follows. In the next section we will
discuss the spine decomposition of branching Markov processes. The
main result, Theorem \ref{maintheorem}, is proved in the last
section.

\section{Spine decomposition}

The materials of this section are mainly taken from \cite{HH1}. We
also refer to \cite{K2} for some materials. The main reason we
present the details here is to clarify some of the points in
\cite{HH1}.

Let ${\mathbb N}=\{1, 2, \dots\}$. We will use
$$
\Gamma:= \bigcup_{n=0}^{\infty}\mathbb{N}^n
$$
(where $\mathbb{N}^0=\{\emptyset\}$) to describe the genealogical
structure of our branching Hunt process. The length (or
generation) $|u|$ of each $u\in\mathbb{N}^n$ is defined to be $n$.
When $n\ge 1$ and $u=(u_1, \dots, u_n)$, we denote $(u_1, \dots,
u_{n-1})$ by $u-1$ and call it the parent of $u$. For each $i\in
\mathbb{N}$ and $u=(u_1, \dots, u_n)$, we write $ui=(u_1, \dots,
u_n, i)$ for the $i$-th child of $u$.
More generally, for $u=(u_1, \dots, u_n), v=(v_1, \dots, v_m)\in \Gamma$,
we will use $uv$ to stand for the
concatenation $(u_1, \dots, u_n, v_1, \dots, v_m)$ of $u$ and $v$.
We will use the notation
$v<u$ to mean that $v$ is an ancestor of $u$. The set of all
ancestors of $u$ is given by $\{v\in \Gamma: \ v<u\}=\{v\in
\Gamma: \exists\ w\in \Gamma\setminus\{\emptyset\}\mbox{ such that
} vw=u\}.$
The notation $v\le u$ has the obvious meaning that either $v<u$ or $v=u$.

A subset $\tau\subset \Gamma$ is called a Galton-Watson tree if a)
$\emptyset\in\tau$; b) if $u,\ v\in\Gamma,$ then $uv\in\tau$ implies
$u\in\tau$; c) for all $u\in\tau,$ there exists $r_u\in\mathbb{N}$
such that when $j\in\mathbb{N},\ uj\in\tau$ if and only if $1\leq
j\leq r_u$. We will denote the collection of Galton-Watson trees by
$\mathbb{T}$. Each $u\in\tau$ is called a node of $\tau$ or an
individual in $\tau$ or just a particle.

To fully describe the branching Hunt process $X$, we need to
introduce the concept of marked Galton-Watson trees. We suppose that
each individual $u\in\tau$ has a mark $(Y_u,\ \sigma_u,\ r_u)$
where:
\begin{itemize}
\item[(i)] $\sigma_u$ is the lifetime of $u$, which determines the
fission time or the death time of particle $u$ as
$\zeta_u=\sum_{v\leq u}\sigma_{v}\ (\zeta_\emptyset =
\sigma_\emptyset)$, and the birth time of $u$ as
$b_u=\sum_{v<u}\sigma_{v}$ $(b_\emptyset =0)$;

\item[(ii)]
$Y_u: [b_u,\ \zeta_u)\rightarrow E_\Delta$ gives the location of
$u$. Given $Y_{u-1}(\zeta_{u-1}-)$ and $b_u$, $(Y_u, u\in [b_u,\
\zeta_u))$ is the restriction to $[b_u,\ \zeta_u)$ of a copy of a
Hunt process starting from $Y_{u-1}(\zeta_{u-1}-)$ at time $b_u$,
i.e., a process with law $\Pi_{Y_{u-1}(\zeta_{u-1}-)}$ shifted by
$b_u$.

\item[(iii)]$r_u $ gives the number
of the offspring born by $u$ when it dies.  It is distributed as
$P(Y_u(\zeta_u-))=(p_k(Y_u(\zeta_u-)))_{k\in\mathbb{N}}$ which is as
defined in Section 1.
\end{itemize}

We will use $(\tau,\ Y, \sigma,\ r )$ (or simply $(\tau, M)$) to
denote a marked Galton-Watson tree. We denote the set of all marked
Galton-Watson trees by $\mathcal{T}=\{(\tau,M):
\tau\in\mathbb{T}\}.$

For any $\tau\in\mathbb{T},$ we can select a line of decent
$\xi=\{\xi_0=\emptyset,\ \xi_1, \xi_2,\cdots\},$ where
$\xi_{n+1}\in\tau$ is an offspring of $\xi_n\in\tau,\ n=0,1, \cdots$.
Such a genealogical line is called a spine.  We will write $(M,
\xi)$ for a marked spine.  We will write $u\in \xi$ to mean that
$u=\xi_i$ for some $i\geq 0$.  We will use
$$
\widetilde{\mathcal{T}}=\{(\tau, Y, \sigma, r, \xi): \
\xi\subset\tau\in\mathbb{T}\}
$$
denote the set of marked trees with distinguished spines.
$L_t=\{u\in \tau: b_u\leq t<\zeta_u\}$ is the set of particles that
are alive at time $t$.

We will use $\widetilde{Y}=(\widetilde{Y}_t)_{t\geq 0}$ to denote
the spatial path followed by a spine and $n=(n_t:  \ t\geq 0)$ to
denote the counting process of fission times along the spine. More
precisely, $\widetilde Y_t=Y_u(t)$ and $n_t=|u|$, if $u\in L_t\cap
\xi.$ We use $\mbox{node}_t((\tau, M,\xi))$, or simply
$\mbox{node}_t(\xi)$, to denote the node in the spine that is alive
at time $t$:
$$
\mbox{node}_t(\xi):=\mbox{node}_t((\tau, M,\xi)):= u \quad \mbox{ if
} u\in\xi\cap L_t.
$$ It is clear that $\mbox{node}_t(\xi)=\xi_{n_t}$.

If $v\in\xi$, then at the fission time $\zeta_v$, it gives birth to
$r_v$ offspring, one of which continues the spine whilst the others
go off to create sub-trees which are copies of the original
branching Hunt process. Let $O_v$ be the set of offspring of $v$
except the one belonging to the spine, then for any $j=1, \dots,
r_v$ such that $vj\in O_v$, we will use $(\tau,\ M)^v_j$ to denote
the marked tree rooted at $vj$.

Now we introduce five filtrations on $\widetilde\cT$ that we shall
use. Define
$$
\begin{array}{rl}\cF_t:=&
\sigma\left\{\left[u, r_u,\sigma_u, (Y_u(s), s\in [b_u, \zeta_u)):
u\in\tau\in\mathbb{T}\mbox{ with }\zeta_u\le t\right]\mbox{ and
}\right.\\&\left. \left[u, (Y_u(s), s\in[b_u, t]):
u\in\tau\in\mathbb{T}\mbox{ with }t\in[b_u, \zeta_u)\right]\right\};\\
\widetilde\cF_t:=&\sigma(\cF_t, (\mbox{node}_s(\xi), s\le t));\\
\cG_t:=&\sigma(\widetilde Y_s: 0\le s\le t);\\
\widehat\cG_t:=&\sigma(\cG_t,  (\mbox{node}_s(\xi):s\le
t),(\zeta_u, u<\mbox{node}_t(\xi)));\\
\widetilde\cG_t:=&\sigma(\cG_t, (\mbox{node}_s(\xi):s\le t),
(\zeta_u, u<\mbox{node}_t(\xi)),(r_u: u<\mbox{node}_t(\xi))).
\end{array}$$

The filtrations $\cF_t$, $\widetilde\cF_t$, $\cG_t$, and $\widetilde
\cG_t$ were introduced in \cite{HH1},  while the filtration
$\widehat \cG_t$ is newly defined. It is obvious that
$\cG_t\subset\widehat\cG_t\subset \widetilde \cG_t\subset
\widetilde\cF_t$. Set $\cF=\bigcup_{t\geq 0}\cF_t,$
$\widetilde\cF=\bigcup_{t\geq 0}\widetilde\cF_t,$
$\cG=\bigcup_{t\geq 0}\cG_t$, $\widehat\cG=\bigcup_{t\geq
0}\widehat\cG_t$ and $\widetilde\cG=\bigcup_{t\geq
0}\widetilde\cG_t$.

\bigskip

For each $x\in E$, let $P^x$ be the measure on $(\widetilde
\cT,\cF)$ such that the filtered probability space $(\widetilde
\cT, \cF, (\cF_t)_{t\ge 0}, (P^x)_{x\in E})$ is the canonical
model  for $X$, the branching Hunt process in $E$. For detailed
constructions of $P^x$, we refer our readers to \cite{C1},
\cite{C2} and  \cite{N}. As noted by Hardy and Harris \cite{HH1},
it is convenient to consider $P^x$ as a measure on the enlarged
space $\widetilde \cT$, rather than on $\cT$. We shall use $P^x_t$
for the restriction of $P^x$ to ${\cF_t}$.

We need to extend the probability measures $P^x$ to probability
measures $\widetilde P^x$ on $(\widetilde\cT,
\widetilde\cF)$ so that the spine is a single genealogical
line of descent chosen from the underlying tree. We will assume
that at each fission time we make a uniform choice amongst the
offspring to decide which line of descent continues the spine
$\xi$. Then for $u\in \tau$ we have
$$
\mbox{Prob}(u\in\xi)=\prod_{v<u}\frac{1}{r_v}.
$$
It is easy to see that
$$
\sum_{u\in L_t}\prod_{v<u}\frac{1}{r_v}=1.
$$

To define $\widetilde P^x$ we recall the following representation
from \cite{L}.

\begin{thm}\label{decom-f}
Every $\widetilde{\mathcal{F}}_t$-measurable function $f$  can be
written as
 \beq\label{decom}f=\sum_{u\in L_t}f_u(\tau,
M)I_{\{u\in\xi\}},
 \eeq
where $f_u$ is $\mathcal{F}_t$-measurable.
\end{thm}
We define the measure  $\widetilde P^x$ on $\widetilde\cF_t$ by
\begin{eqnarray}\label{spine representation}
\mbox{d}\widetilde P^x(\tau, M, \xi)\Big|_{\widetilde
\cF_t}=\mbox{d}\Pi_x(\widetilde Y) \mbox{d}L^{\beta(\widetilde
Y)}({\bf n})\prod_{v<\xi_{n_t}}p_{r_v}(\widetilde Y_{\zeta_v})
\prod_{v<\xi_{n_t}}\frac{1}{r_v} \prod_{j:\ vj\in
O_v}\mbox{d}P^{\widetilde Y_{\zeta_v}}_{t-\zeta_v}((\tau, M)^v_j),
\end{eqnarray}
where $L^{\beta(\widetilde Y)}({\bf n})$ is the law of the Poisson
random measure ${\bf n}=\{\{\sigma_i: i=1, \cdots, n_t\}: t\ge 0\} $
with intensity $\beta(\widetilde Y_t)dt$ along the path of
$\widetilde Y$, $\Pi_x(\widetilde Y)$ is the law of the diffusion
$\widetilde Y$ staring from $x\in E$, and $p_{r_v}(y)=\sum_{k\ge
2}p_k(y)I_{(r_v=k)}$ is the probability that individual $v$, located
at $y\in E$, has $r_v$ offspring.

It follows from Theorem \ref{decom-f} that for any
bounded $f\in \widetilde\cF_t$,
$$\begin{array}{rl}
\widetilde P^x(f|\cF_t)
=&\displaystyle\widetilde P^x\left(\left.\sum_{u\in L_t}f_u(\tau,
M)I_{\{u\in\xi\}}\right|\cF_t\right)\\
=&\displaystyle\sum_{u\in L_t}f_u(\tau, M)\widetilde
P^x\left(\left.I_{\{u\in\xi\}}\right|\cF_t\right)\\
=&\displaystyle\sum_{u\in L_t}f_u(\tau, M)\prod_{v<u}\frac{1}{r_v}.
\end{array}$$
Then we have
 \beq\label{decom-P} \widetilde
P^x(f)=P^x\left(\sum_{u\in L_t}f_u(\tau,
M)\prod_{v<u}\frac{1}{r_v}\right),\quad \mbox{ for any bounded
}f\in\widetilde \cF_t, t\ge 0.
 \eeq
In particular,
$$
\widetilde P^x(\widetilde \cT)=P^x\left(\sum_{u\in
L_t}\prod_{v<u}\frac{1}{r_v}\right)= P^x(1)=1,
$$
which implies $\widetilde P^x$ is a probability measure. $\widetilde
P^x$ is an extension of $P^x$ onto $(\widetilde\cT,\widetilde\cF)$
and for any bounded $f\in\widetilde\cF_t$ we have
\begin{equation}\label{many-to-one}
\int_{\widetilde\cT}f\ \mbox{d}\widetilde
P^x=\int_{\widetilde\cT}\sum_{u\in
L_t}f_u\prod_{v<u}\frac{1}{r_v}\ \mbox{d}P^x.
\end{equation}

The decomposition \eqref{spine representation} of $\widetilde P^x$
suggests the following intuitive construction of the system under
$\widetilde P^x$:
\begin{itemize}
\item[(i)] the root of $\tau$ is at $x$, and the spine process
$\widetilde Y_t$ moves according to the measure $\Pi_x$;

\item[(ii)] given the trajectory $\widetilde{Y_\cdot}$ of the spine,
the fission time $\zeta_v$ of node $v$ on the spine is distributed
according to $L^{\beta(\widetilde{Y})},$ where
$L^{\beta(\widetilde{Y})}$ is the law of the Poisson random measure
with intensity $\beta(\widetilde Y_t)dt$;

\item[(iii)] at the fission time of node $v$ on the spine, the single
spine particle is replaced by  a random number $r_v$ of offspring
with $r_v$ being distributed according to the law $P(\widetilde
Y_{\zeta_v})=(p_k(\widetilde Y_{\zeta_v}))_{k\ge 1}$;

\item[(vi)] the spine is chosen uniformly from the $r_v$ offspring of $v$ at
the fission time of $v$;

\item[(v)] each of the remaining $r_v-1$ particles $vj\in O_v$ gives rise to the
independent subtrees $(\tau, M)^{v}_j$, which evolve as independent
subtrees determined by the probability measure $P^{\widetilde
Y_{\zeta_v}}$ shifted to the time of creation.
\end{itemize}

\begin{definition}\label{con-mart}
Suppose that $(\Omega, \cH, P)$ is a probability space,
$\{\cH_t,t\ge 0\}$ is a filtration on $(\Omega, \cH)$ and that
${\cal K}$ is a sub-$\sigma$-field of $\cH$. A real-valued process
$U_t$ on $(\Omega, \cF, P)$ is called a $P(\cdot |\ {\cal
K})$-martingale with respect
to $\{\cH_t,t\ge 0\}$ if (i) it is adapted to $\{\cH_t\vee{\cal
K},t\ge 0\}$; (ii) for any $t\geq 0,\ E(|U_t|)<\infty$ and (iii) for
any $t>s$,
$$
E(U_t\big|\cH_s\vee{\cal K})=
U_s,\quad{\rm a.s.}
$$
We also say that $U_t$ is a martingale
with respect to $\{\cH_t,t\ge 0\}$, given
${\cal K}$.
\end{definition}

\begin{lemma}\label{mart-prod}
Suppose that $(\Omega, \cH, P)$ is a probability space,
$\{\cH_t,t\ge 0\}$ is a filtration on $(\Omega, \cH)$ and that
${\cal K}_1, {\cal K}_2$ are two sub-$\sigma$-fields of ${\cal H}$
such that ${\cal K}_1\subset{\cal K}_2$. Assume that $U^1_t$ is a
$P(\cdot |\ {\cal K}_1)$-martingale with respect to $\{{\cal
H}_t,t\ge 0\}$, $U^2_t$ is a $P(\cdot |{\cal K}_2)$-martingale with
respect to $\{{\cal H}_t,t\ge 0\}$. If $U^1_t\in{\cal K}_2$,
$U^2_t\in\cH_t$, and $E\left(|U_t^1U_t^2|\right)<\infty$ for any
$t\ge 0$, then the product $U^1_t U^2_t$ is a $P(\cdot |\ {\cal
K}_1)$-martingale with respect to $\{\cH_t,t\ge 0\}$.
\end{lemma}

\begin{proof} Suppose that $t\ge s\ge 0$.
The assumption that $U^1_t\in{\cal K}_2$ implies that $U^1_t\in
\cH_s\vee{\cal K}_2$. Then
$$
\begin{array}{rl}P(U^1_tU^2_t| \cH_s\vee{\cal K}_1)=&
P\left[P(U^1_tU^2_t|\cH_s\vee{\cal K}_2)|\cH_s\vee {\cal
K}_1\right]\\
=&P\left[U^1_tP(U^2_t|\cH_s\vee{\cal K}_2)|\cH_s\vee
{\cal K}_1\right]\\
=&P\left[U^1_tU^2_s|\cH_s\vee {\cal K}_1\right]\\
=&U^2_sP\left[U^1_t|\cH_s\vee {\cal K}_1\right]\\
=& U^1_sU^2_s,\end{array}
$$
where in the last second equality we used the assumption that
$U^2_s\in{\cal H}_s$.\qed
\end{proof}

\begin{lemma}\label{HH1-results}
Suppose that ${\bf n}=\{\{\zeta_i: i=1, \cdots, n_t\}: t\ge 0\} $ is
a Poisson random  measure with intensity $\beta(\widetilde Y_t)dt$
along the path of $\widetilde Y$.  Then
$$
\eta_t^{(1)}:=\prod_{i\le n_t}A(\widetilde
Y_{\zeta_i})\cdot\exp\left(-\int^t_0((A-1)\beta)(\widetilde
Y_s)ds\right)
$$
is an $L^{\beta(\widetilde Y)}$-martingale with respect to the
natural filtration $\{\cL_t\}$ of  ${\bf n}$.
\end{lemma}

\begin{proof}  First note  that
 \beq\label{Possion}
L^{\beta(\widetilde Y)}\left[\prod_{i\le n_t}A(\widetilde
Y_{\zeta_i})\right]=\exp\left(\int^t_0((A-1)\beta)(\widetilde
Y_s)ds\right),
 \eeq
which implies that  $L^{\beta(\widetilde Y)}(\eta^{(1)}_t)=1$. It is
easy to check that $\eta_t^{(1)}$ is a martingale under
$L^{\beta(\widetilde Y)}$ by using the Markov property of ${\bf n}$.
We omit the details. \qed
\end{proof}

It follows from the lemma above that we can define a measure
$L^{(A\beta)(\widetilde Y)}$ by
$$
\frac{dL^{(A\beta)(\widetilde Y)}}{dL^{\beta(\widetilde Y)}}
\Bigg|_{\cL_t} = \prod_{i \le {n_t}}A(\widetilde
Y_{\zeta_i})\cdot\exp\left(-\int^t_0((A-1)\beta)(\widetilde
Y_s)ds\right).
$$

\begin{lemma}
For any $x\in E$ and $t\ge 0$, we have
\beq\label{cond-mean1}\displaystyle\widetilde
P^x\left[\left.\prod_{v<\xi_{n_t}}\frac{r_v}{A(\widetilde
Y_{\zeta_v})}\right|\widehat\cG\right]=1.\eeq
\end{lemma}

\begin{proof}
It follows from \eqref{spine representation} that, given
$\widehat\cG$, for each $v<\xi_{n_t}$,
$$
\widetilde P^x(r(\widetilde
Y_{\zeta_v})|\widehat\cG)=A(\widetilde Y_{\zeta_v}).
$$
Since, given $\widehat\cG$,  $\{r_v, v<\xi_{n_t}\}$ are independent,
we have
$$
\widetilde P^x \left(\left.\prod_{v<\xi_{n_t}}\frac{r(\widetilde
Y_{\zeta_i})}{A(\widetilde Y_{\zeta_i})}\right|\widehat\cG\right)=1.
$$
\qed
\end{proof}

The following lemma corresponds to Theorems 5.4 and 5.5 in
\cite{HH1} which were not proved there. Our results are somewhat
different from those stated in Theorems 5.4 and 5.5 in \cite{HH1}.

\begin{lemma}\label{HH1-results}
(1) The process
$$
\widetilde\eta_t^{(1)}:=\prod_{v<\xi_{n_t}}A(\widetilde
Y_{\zeta_v})\cdot\exp\left(-\int^t_0((A-1)\beta)(\widetilde
Y_s)ds\right)
$$
is a $\widetilde P^x(\cdot |\ {\cG})$-martingale with respect to
$\{\widetilde\cF_t,t\ge 0\}$.

(2) The process
$$
\widetilde\eta_t^{(2)}:=\prod_{v<\xi_{n_t}}\frac{r_v}{A(\widetilde
Y_{\zeta_v})}
$$
is  a $\widetilde P^x(\cdot |\widehat\cG)$-martingale with respect
to $\{\widetilde\cF_t,t\ge 0\}$.
\end{lemma}
\begin{proof}  (1) For $s,t\ge 0$, by the Markov property, we have
$$\begin{array}{rl}\widetilde
P^x\left[\left.\widetilde\eta^{(1)}_{t+s}\right|\widetilde\cF_t\vee\cG\right]=
&\displaystyle\widetilde
P^x\left[\left.\prod_{v<\xi_{n_{t+s}}}A(\widetilde Y_{\zeta_v})
\cdot\exp\left(-\int^{t+s}_0((A-1)\beta)(\widetilde
Y_r)dr\right)\right|\widetilde\cF_t\vee\cG\right]\\
=&\displaystyle\prod_{v<\xi_{n_t}}A(\widetilde Y_{\zeta_v})\cdot
\exp\left(-\int^t_0((A-1)\beta)(\widetilde Y_r)dr\right)\cdot\\
&\displaystyle\widetilde P^{x}\left[\left.\prod_{\xi_{n_t}\le
v<\xi_{n_{t+s}}}A(\widetilde Y_{\zeta_v})\cdot
\exp\left(-\int^{s}_0((A-1)\beta)(\widetilde
Y_{r+t})dr\right)\right|\widetilde\cF_t\vee\cG\right]\\
=&\displaystyle\widetilde\eta_t^{(1)}\exp\left(-\int^{s}_0((A-1)\beta)(\widetilde
Y_{r+t})dr\right)\displaystyle\widetilde
P^x\left[\left.\prod_{\xi_{n_{t}}\le v<\xi_{n_{t+s}}}A(\widetilde
Y_{\zeta_v}) \right|\cG\right].
\end{array}$$
For fixed $t>0$, given the path of $\widetilde Y$, the collection of
fission times $\{\{\zeta_v: \xi_{n_{t}}\le v<\xi_{n_{t+s}}\}: s\ge
0\} $ is a Poisson random measure  with intensity $\beta(\widetilde
Y_{t+s})ds$, and has law $L^{\beta(\widetilde Y_{t+\cdot})}$. It
follows from \eqref{Possion} that
$$P^x\left[\left.\prod_{\xi_{n_{t}}\le v<\xi_{n_{t+s}}}A(\widetilde
Y_{\zeta_v})
\right|\cG\right]=\exp\left(\int^{s}_0((A-1)\beta)(\widetilde
Y_{r+t})dr\right).$$ Then we get
$$
\widetilde P^x\left[\left.\widetilde\eta^{(1)}_{t+s}
\right|\widehat\cF_t\vee\cG\right]=\widetilde\eta_t^{(1)}.
$$

(2) For $s,t\ge 0$, by the Markov property, we have
$$\begin{array}{rl}P^x\left[\left.\widetilde\eta^{(2)}_{t+s}
\right|\widetilde\cF_t\vee\widehat\cG\right]
=&\displaystyle\widetilde
P^x\left[\left.\prod_{v<\xi_{n_{t+s}}}\frac{r_v}{A(\widetilde
Y_{\zeta_v})}\right|\widetilde\cF_t\vee\widehat\cG\right]\\
=& \displaystyle\prod_{v<\xi_{n_t}}\frac{r_v}{A(\widetilde
Y_{\zeta_v})}\cdot\displaystyle\widetilde
P^{x}\left[\left.\prod_{\xi_{n_t}\leq
v<\xi_{n_{s+t}}}\frac{r_v}{A(\widetilde
Y_{\zeta_v})}\right|\widehat\cG\right]\\
=&\displaystyle\widetilde\eta^{(2)}_t,
\end{array}
$$ where in the last
equality we used \eqref{cond-mean1}. Then we have
$$
\widetilde P^x\left[\left.\widetilde\eta^{(2)}_{t+s}\right|
\widetilde\cF_t\vee\widehat\cG\right]=\widetilde\eta^{(2)}_t.
$$
\qed
\end{proof}

The effect of a change of measure using the
martingale $\widetilde\eta_t^{(1)}$ will
increase the fission rate along the spine from $\beta(\widetilde
Y_t)$ to $(A\beta)(\widetilde Y_t)$.
The effect of a change of measure using the martingale
$\widetilde\eta_t^{(2)}$ will change offspring distribution from
$P(\widetilde{Y}_{\zeta_i})=(p_k(\widetilde Y_{\zeta_i}))_{k\ge
1}$ to the size-biased distribution $\widehat
P(\widetilde{Y}_{\zeta_i})=(\hat p_k(Y_{\zeta_i}))_{k\ge 1}$,
where $\hat p_k(y)$ is defined by
$$
\hat p_k(y)=\frac{kp_k(y)}{A(y)},\quad  k\ge 1,y\in E.
$$

Define
$$\widetilde\eta^{(3)}_t(\phi):=\frac{\phi(\widetilde
Y_{t})}{\phi(x)}\exp\left(-\int_0^{t}{(\lambda_1-(A-1)\beta)}(\widetilde
Y_s)ds\right).$$ $\widetilde\eta^{(3)}_t(\phi)$ is a $\widetilde
P^x$-martingale with respect to $\{\cG_t, t\ge 0\}$, and it is also
a $\widetilde P^x$-martingale with respect to $\{\widetilde\cF_t,
t\ge 0\}$, since $\widetilde\eta_t^{(3)}(\phi)$ can be expressed as
\begin{equation}\label{decom-eta}
\widetilde\eta^{(3)}_t(\phi)=\sum_{u\in
L_t}\phi(x)^{-1}\phi(\widetilde
Y_u(t))\exp\left(-\int_0^{t}{(\lambda_1-(A-1)\beta)}(\widetilde
Y_s)ds\right)I_{\{u\in\xi\}}.
\end{equation}
And then we define
$$\begin{array}{rl}\widetilde
\eta_t(\phi)=&\displaystyle\prod_{v<\xi_{n_t}}r_v\exp
\left(-\int^t_0((A-1)\beta)(\widetilde
Y_s)ds\right)\times
\widetilde\eta^{(3)}_t(\phi)\\
=&\displaystyle\prod_{v<\xi_{n_t}}\frac{r_v}{A(\widetilde
Y_{\zeta_v})}\prod_{v<\xi_{n_t}}A(\widetilde
Y_{\zeta_v})\exp\left(-\int^t_0((A-1)\beta)(\widetilde
Y_s)ds\right)\times \widetilde\eta^{(3)}_t(\phi)\\
=&\widetilde\eta_t^{(1)}\times\widetilde\eta_t^{(2)}
\times\widetilde\eta^{(3)}_t(\phi).\end{array}$$

The following result corresponds to Definition 5.6 in \cite{HH1}.

\begin{lemma} $\widetilde
\eta_t(\phi)$ is a $\widetilde P^x$-martingale with respect to
$\widetilde\cF_t$.
\end{lemma}

\begin{proof}
$\widetilde\eta_t^{(1)}$ is a $\widetilde P^x(\cdot |\
{\cG})$-martingale with respect to $\{\widetilde\cF_t,t\ge 0\}$, and
$\widetilde\eta_t^{(2)}$ is a $\widetilde P^x(\cdot
|\widehat\cG)$-martingale with respect to $\{\widetilde\cF_t,t\ge
0\}$. Note that $\cG\subset\widehat\cG$, and
$\widetilde\eta^{(1)}_t\in\widehat \cG,
\widetilde\eta^{(2)}_t\in\widetilde \cF_t$ for any $t\ge 0$. Using
Lemma \ref{mart-prod},
$\widetilde\eta^{(1)}_t\widetilde\eta^{(2)}_t$ is a $\widetilde
P^x(\cdot|\cG)$-martingale with respect to $\{\widetilde\cF_t,t\ge
0\}$. Note that $\widetilde\eta^{(3)}_t(\phi)\in \cG$ and
$\widetilde\eta^{(1)}_t\widetilde\eta^{(2)}_t\in\widetilde \cF_t$
for any $t\ge 0$. Using Lemma \ref{mart-prod} again, we see that
$\widetilde
\eta_t(\phi)=\widetilde\eta^{(1)}_t\widetilde\eta^{(2)}_t\widetilde\eta^{(3)}_t(\phi)$
is a $\widetilde P^x$-martingale with respect to
$\{\widetilde\cF_t,t\ge 0\}$.
\qed\end{proof}

\begin{lemma} $M_t(\phi)$ is the projection of $\widetilde\eta_t(\phi)$
onto $\cF_t$, i.e.,
$$M_t(\phi)=\widetilde P^x(\widetilde\eta_t(\phi)|\cF_t).$$
\end{lemma}
\begin{proof} By \eqref{decom-eta}, we have
$$
\widetilde\eta_t(\phi)=\sum_{u\in L_t}\prod_{v<u}r_ve^{-\lambda_1
t}\phi(x)^{-1}\phi(Y_u(t))I_{(u\in\xi)}.
$$
Then
$$
\begin{array}{rl}\widetilde P^x(\widetilde\eta_t(\phi)|\cF_t)
=&\displaystyle\sum_{u\in
L_t}e^{-\lambda_1t}\phi(x)^{-1}\phi(Y_u(t))\prod_{v<u}r_v\
\widetilde P^x(I_{(u\in\xi)}|\cF_t)\\
=&\displaystyle\sum_{u\in
L_t}e^{-\lambda_1t}\phi(x)^{-1}\phi(Y_u(t))=M_t(\phi),
\end{array}
$$
where in the second equality we used the fact that $\widetilde
P^x(I_{(u\in L_t \cap\xi)}|\cF_t)=I_{(u\in
L_t)}\times\prod_{v<u}r_v^{-1}.$\qed
\end{proof}

Now we define a probability measure $\widetilde Q^x$ on
$(\widetilde\cT, \widetilde\cF)$ by
$$
\frac{d \widetilde Q^x}{d \widetilde
P^x}\left|_{\widetilde\cF_t}\right.=\widetilde\eta_t(\phi),
$$ which
says that on $\widetilde\cF_t$,
$$
\begin{array}{rl}d \widetilde Q^x=&\displaystyle\widetilde
\eta_t(\phi)d \widetilde P^x\\
=&\displaystyle \frac{\phi(\widetilde
Y_{t\wedge\tau^B})}{\phi(x)}\exp\left(-\int_0^{t\wedge\tau^B}
(\lambda_1-(A-1)\beta)(\widetilde
Y_s)ds\right)d\Pi_x(\widetilde Y)\\
&\displaystyle\times \exp\left(-\int^t_0((A-1)\beta)(\widetilde
Y_s)ds\right) dL^{\beta(\widetilde Y)}
\prod_{v<\xi_{n_t}}p_{r_v}(\widetilde Y_{\zeta_v}) \prod_{j:\ vj\in
O_v}dP^{\widetilde
Y_{\zeta_v}}_{t-\zeta_v}((\tau, M)^v_j)\\
=&\displaystyle d\Pi^\phi_x(\widetilde Y) dL^{A\beta(\widetilde
Y)}({\bf n}) \prod_{v<\xi_{n_t}} \frac{p_{r_v}(\widetilde
Y_{\zeta_v})}{A(\widetilde Y_{\zeta_v})} \prod_{j:\ vj\in
O_v}dP^{\widetilde
Y_{\zeta_v}}_{t-\zeta_v}((\tau, M)^v_j)\\
=&\displaystyle d\Pi^\phi_x(\widetilde Y) dL^{A\beta(\widetilde
Y)}({\bf n}) \prod_{v<\xi_{n_t}}  \hat p_{r_v}(\widetilde
Y_{\zeta_v})\prod_{v<\xi_{n_t}}\frac{1}{r_v} \prod_{j:\ vj\in
O_v}dP^{\widetilde
 Y_{\zeta_v}}_{t-\zeta_v}((\tau,
M)^v_j).\end{array}
$$
Thus the change of measure from $\widetilde P^x$ to $\widetilde Q^x$
has three effects: the spine will be changed to a Hunt process under
$\Pi^\phi_x$, its fission times will be increased and the
distribution of its family sizes will be sized-biased. More
precisely, under $\widetilde Q^x$,

(i) the spine process $\widetilde Y_t$ moves according to the
measure $\Pi^\phi_x$;

(ii)  the fission times along the spine occur at an accelerated
intensity $(A\beta)(\widetilde Y_t)dt$;

(iii) at the fission time of node $v$ on the spine, the single spine
particle is replaced by  a random number $r_v$ of offspring with
size-biased offspring distribution $\widehat P(\widetilde
Y_{\zeta_v}):=\left(\hat p_k(\widetilde Y_{\zeta_v})\right)_{k\ge
1}$, where $\hat p_k(y)$ is defined by $\hat p_k(y):=\frac{k
p_k(y)}{A(y)}$, $k=1, 2.\cdots$,  $y\in E$;

(iv) the spine is chosen uniformly from the $r_v$ particles at the
fission point $v$;

(v) each of the remaining $r_v-1$ particles $vj\in O_v$ gives rise
to independent subtrees $(\tau, M)^{v}_j$  which evolve as
independent subtrees determined by the probability measure
$P^{\widetilde Y_{\zeta_v}}$ shifted to the time of creation.

We define a measure $Q^x$ on $(\widetilde\cT, \cF)$ by
$$
Q^x:=\widetilde Q^x|_{\cF}.
$$
It follows from Theorem 6.4 in \cite{HH1} and its proof that $Q^x$
is a martingale change of measure by the martingale $M_t(\phi)$:
$$
\left.\frac{d Q^x}{d P^x}\right|_{{\cal F}_t}=M_t(\phi).
$$

\begin{thm}[Spine decomposition]
We have the following spine decomposition for the martingale $M_t(\phi)$:
\begin{equation}\label{spine-decom}
\widetilde Q^x\left[\phi(x)M_t(\phi)\Big|\widetilde{\cG}\right] =
\phi(\widetilde Y_t)e^{-\lambda_1t}+
\sum_{u<\xi_{n_t}}(r_u-1)\phi(\widetilde
Y_{\zeta_u})e^{-\lambda_1\zeta_u}.
\end{equation}
\end{thm}
\begin{proof}
We first decompose the martingale $\phi(x)M_t(\phi)$ as
$$
\phi(x)M_t(\phi)=e^{-\lambda_1t}\phi(\widetilde Y_{t})
+e^{-\lambda_1t}\sum_{u\in L_t, u\neq\xi_{n_t}}\phi(Y_u(t)).
$$
The individuals $\{u\in L_t, u\neq\xi_{n_t}\}$ can be partitioned
into subtrees created from fissions along the spines. That is, each
node $u<\xi_{n_t}$ in the spine $\xi$ has given birth at time
$\zeta_u$ to $r_u$ offspring among which one has been chosen as a
node of the spine whilst the other $r_u-1$ individuals go off to
make the subtree $(\tau, M)^u_j$.  Put
$$
X^j_t=\sum_{v\in L_t, v\in(\tau, M)^u_j}\delta_{Y_v(t)}(\cdot),\quad
t\ge \zeta_u.
$$
$(X^j_t, t\ge \zeta_u)$ is a $(Y, \beta, \psi)$-branching Hunt
process with birth time $\zeta_u$ and staring point $\widetilde
Y_{\zeta_u}$.  Then
 \begin{equation}\label{decom}\phi(x)M_t(\phi)=e^{-\lambda_1t}\phi(\widetilde
 Y_{t})+\sum_{u<\xi_{n_t}}
 \sum_{j:\ uj\in O_u}M_{t}^{u,j}(\phi)\phi(\widetilde{Y}_{\zeta_u})e^{-\lambda_1\zeta_u},\end{equation}
where
$$
M^{u,j}_{t}(\phi):=e^{-\lambda_1
(t-\zeta_u)}\frac{\langle\phi, X_{t-\zeta_u}^
j\rangle}{\phi(\widetilde{Y}_{\zeta_u})}
$$
is, conditional on
$\widetilde\cG$, a $\widetilde P^x$-martingale on the subtree
$(\tau, M)^u_j$, and therefore
$$\widetilde P^x(M^{u,j}_{t}(\phi)|\widetilde\cG)=1.$$
Thus taking $\widetilde Q^x$ conditional expectation of \eqref{decom}
gives
$$
\widetilde Q^x\left[\phi(x)M_t^\phi\Big|\widetilde{\cG}\right]
 = \phi(\widetilde Y_t)e^{-\lambda_1t}+
 \sum_{u<\xi_{n_t}}(r_u-1)\phi(\widetilde Y_{\zeta_u})e^{-\lambda_1\zeta_u},
$$
which completes the proof.\qed
\end{proof}

\bigskip

\section{Proof of the main result}
First, we give two lemmas.  The first lemma is basically
\cite[Theorem 4.3.3]{D2}.
\begin{lemma}\label{Durrett2}
 Suppose that $\P$ and $\Q$ are two probability measures on a
 measurable  space
$(\Omega, {\cal F}_{\infty})$ with filtration $({\cal F}_t)_{t\ge
0}$, such that for some nonnegative martingale $Z_t$,
$$\frac{d\Q}{d\P}\Big|_{{\cal F}_t}=Z_t.$$
The limit $Z_{\infty}:=\limsup_{t\to\infty}Z_t$ therefore exists and
is finite almost surely under $\P$. Furthermore, for any $F\in{\cal
F}_{\infty}$
$$\Q(F)=\int_FZ_{\infty}d\P+\Q(F\cap\{Z_{\infty}=\infty\}),$$
and consequently,
$$\begin{array}{rl}&(a)\quad \P(Z_{\infty}=0)=1\Longleftrightarrow
\Q(Z_{\infty}=\infty)=1\\
&(b)\quad \displaystyle\int Z_{\infty}d\P=\displaystyle\int
Z_0d\P\Longleftrightarrow\Q(Z_{\infty}<\infty)=1.\end{array}$$
\end{lemma}

Now we are going to give a lemma which is the key to the proof of
Theorem \ref{maintheorem}. To state this lemma, we need some  more
notation.  Note that under $\widetilde Q^x$, given $\widetilde\cG$,
$N_t:=\{\{(\zeta_{\xi_i},\ r_{\xi_i}): i=0, 1,2, \cdots,
n_t-1\}:t\geq 0\}$ is a Poisson point process with instant intensity
measure $(A\beta)(\widetilde Y_t)dtd\widehat P(\widetilde Y_t )$ at
time $t$, where for each $y\in E$, $\widehat P(y)$ is the
size-biased probability measure on $\N$ defined in Lemma
\ref{HH1-results}. To simplify notation, $\zeta_{\xi_i}$ and
$r_{\xi_i}$ will be denoted as $\zeta_i$ and $r_i$, respectively.

Recall that $l(x)=\sum_{i=2}^\infty (i\phi(x))\log^+( i\phi(x))\,
p_i(x).$

\begin{lemma}\label{lemma1}

(1) If $\int_E\widetilde{\phi}(y)\beta(y)l(y)m(dy)<\infty$, then
$$
\sum^\infty_{i=0}{e}^{-\lambda_1\zeta_i}r_i
\phi(\widetilde{Y}_{\zeta_i})<\infty, \quad \widetilde
Q^x-\mbox{a.s.}
$$

 (2) If $
\int_E\widetilde{\phi}(y)\beta(y)l(y)m(dy)=\infty$,  then
$$
\limsup_{i\rightarrow\infty}e^{-\lambda_1\zeta_i}r_i\phi(
\widetilde{Y}_{\zeta_i})=\infty,\quad \widetilde Q^x-\mbox{a.s.}
$$
\end{lemma}
\begin{proof}
(1) For any $\epsilon>0$,
\begin{eqnarray}\label{sum}
&&\sum^{\infty}_{i=0}{e}^{-\lambda_1\zeta_i}r_i\phi(\widetilde{Y}_{\zeta_i})\nonumber\\
&=&\sum_i e^{-\lambda_1\zeta_i}r_i\phi(\widetilde{Y}_{\zeta_i})
I_{\{r_i\phi(\widetilde{Y}_{\zeta_i})\le e^{\varepsilon
\zeta_i}\}}+\sum_i
e^{-\lambda_1\zeta_i}r_i\phi(\widetilde{Y}_{\zeta_i})
I_{\{r_i\phi(\widetilde{Y}_{\zeta_i})>e^{\varepsilon \zeta_i}\}}\nonumber\\
&=:&I+II.\end{eqnarray}
Note that \eqref{IU} implies that there is a constant
 $c>0$ such that for any $t>c$ and any nonnegative
measurable function $f$ with $||f||_{\infty}\leq 1,$
\begin{equation}\label{domi-p}
\frac{1}{2}\int_E\phi(y)\widetilde{\phi}(y)f(y)m(dy)\leq
\int_Ep^{\phi}(t, x, y)f(y)m(dy)\leq 2
\int_E\phi(y)\widetilde{\phi}(y)f(y)m(dy),\quad x\in E.
\end{equation}
Then,
\begin{eqnarray}\label{II}
&&\widetilde Q^x\left[\sum_{i}
I_{\left\{r_i\phi(\widetilde{Y}_{\zeta_i})> e^{
\varepsilon\zeta_i}\right\}}\right]\nonumber\\
&=&\Pi^{\phi}_x\left[\int_0^\infty \beta(\widetilde Y_s)
\left(\sum_{k=2}^\infty kp_k(\widetilde{Y}_s)I_{\{k\phi(
\widetilde Y_s)> e^{\epsilon s}\}}\right)ds\right]\nonumber\\
&=&\int_0^\infty ds\int_E {p^\phi(s,x,\ y)m(dy)}\left[\beta(y)
\sum_{k=2}^\infty kp_k(y)I_{\{k\phi(y)>e^{\varepsilon s}\}}\right]\nonumber\\
&=&\int_0^cds\int_E {p^\phi(s,x,\ y)m(dy)}\left[\beta(y)
\sum_{k=2}^\infty kp_k(y)I_{\{k\phi(y)>e^{\varepsilon s}\}}\right]\nonumber\\
 &&+\int_c^\infty ds\int_E {p^\phi(s,x,\ y)m(dy)}\left[\beta(y)
\sum_{k=2}^\infty kp_k(y)I_{\{k\phi(y)>e^{\varepsilon s}\}}\right]
\end{eqnarray}
Applying Fubini's theorem and using the assumption that $A$ and $\beta$
are bounded, we get
\begin{equation}\label{II-1}\int_0^cds\int_E {p^\phi(s,x,\ y)m(dy)}\left[\beta(y)
\sum_{k=2}^\infty kp_k(y)I_{\{k\phi(y)>e^{\varepsilon s}\}}\right]\le
\|\beta A\|_{\infty}\int_0^cds\le C_1,\end{equation}
where $C_1$ is positive constant which only depends on $c$. Using \eqref{domi-p}, we get
\begin{eqnarray}\label{II-2}&&\displaystyle\int_c^\infty ds
\int_E {p^\phi(s,x,\ y)m(dy)}\left[\beta(y)
\sum_{k=2}^\infty kp_k(y)I_{\{k\phi(y)>e^{\varepsilon s}\}}\right]\nonumber\\
&\le&2\displaystyle\int_E
m(dy)\phi(y)\widetilde{\phi}(y)\beta(y)\sum_{k=2}^\infty
kp_k(y)\int_0^{\frac{1}{\varepsilon}\log^+ [k\phi(y)]}ds\nonumber\\
&=&\displaystyle\frac{2}{\varepsilon}\int_E\widetilde{\phi}(y)\beta(y)l(y)m(dy).\end{eqnarray}
Combining \eqref{II}, \eqref{II-1} and \eqref{II-2}, we have
\begin{equation*}
\widetilde Q^x\left[\sum_{i}
I_{\left\{r_i\phi(\widetilde{Y}_{\zeta_i})> e^{
\varepsilon\zeta_i}\right\}}\right]
\le C_1+\frac{2}{\varepsilon}\int_E\widetilde{\phi}(y)\beta(y)l(y)m(dy).
\end{equation*}
Therefore, the condition
$\int_E\widetilde{\phi}(y)\beta(y)l(y)m(dy)<\infty$ implies that
\[\sum_{i}I_{\left\{r_i\phi(\widetilde{Y}_{\zeta_i})>
e^{\varepsilon\zeta_i}\right\}}<\infty,\quad
\widetilde Q^x-\mbox{a.s.}\]
  for all $\varepsilon>0.$  Then we have
\begin{equation}\label{big}
II<\infty,\quad \widetilde Q^x-\mbox{a.s.}
\end{equation}

Meanwhile for $\varepsilon<\lambda_1$,
\begin{eqnarray*}
\widetilde Q^x(I) &=&\widetilde
Q^x\left[\sum_{i}{e}^{-\lambda_1\zeta_i}r_i\phi(\widetilde
Y_{\zeta_i})I_{\{r_i\phi(\widetilde{Y}_{\zeta_i})\le{e}^{\varepsilon\zeta_i}\}}\right]\\
&=&\Pi_x^\phi\int_0^\infty dt e^{-\lambda_1
t}\phi(\widetilde{Y}_t)\beta(\widetilde{Y}_t)A(\widetilde Y_t)\sum_{k=2}^\infty
k\hat p_k(\widetilde Y_t)
I_{\{k\phi(\widetilde{Y}_t)\leq e^{\varepsilon t}\}}\\
&=&\Pi_x^\phi\int_0^\infty dt e^{-\lambda_1
t}\phi(\widetilde{Y}_t)\beta(\widetilde{Y}_t)A(\widetilde Y_t)\sum_{k=2}^\infty
k\frac{k}{A(\widetilde Y_t)}p_k(\widetilde Y_t)
I_{\{k\phi(\widetilde{Y}_t)\leq e^{\varepsilon t}\}}\\
&\leq &\Pi_x^\phi\int_0^\infty dt
e^{-(\lambda_1-\varepsilon)t}\beta(\widetilde{Y}_t)\sum_{k=2}^\infty
kp_k(\widetilde Y_t)I_{\{k\phi(\widetilde{Y}_t)\le e^{\varepsilon
t}\}}\\
&\leq & C_2\int_0^\infty e^{-(\lambda_1-\varepsilon)t}dt<\infty,
\end{eqnarray*}
where in the second to the last inequality we used the assumption
that $\beta$ and $A$ are bounded and $C_2$ is a positive
constant. Then we have
\begin{equation}\label{small}
I<\infty,\quad \widetilde Q^x-{\rm a.s.}
\end{equation}
 Combining \eqref{sum}, \eqref{big} and \eqref{small}, we see that
$\sum^{\infty}_{i=0}{e}^{-\lambda_1\zeta_i}r_i\phi(
\widetilde{Y}_{\zeta_i})<\infty,
\quad \widetilde Q^x-\mbox{a.s.}$\\

(2) It is enough to prove that for any $K>1$,
\begin{equation}\label{>K}
\limsup_{i\rightarrow\infty}{e}^{-\lambda_1\zeta_i}r_i\phi(
\widetilde{Y}_{\zeta_i})>K,\quad
\widetilde Q^x-{\rm a.s.}
\end{equation}
For any fixed $K>1$, define  $\gamma(t,y):=\beta(y)\sum_k
kp_k(y)I_{\{k\phi(y)>K{e}^{\lambda_1t}\}}$.
Since for any $x\in E$,
\begin{eqnarray*}
\Pi^\phi_x \int_0^T \gamma(t,\widetilde{Y}_t) dt
&=&\int^T_0dt\int_E m(dy) p^{\phi}(t, x, y)\gamma(t, y)\\
&\le&\int^T_0dt\int_E m(dy) p^{\phi}(t, x, y)\beta(y)A(y)<\infty,
\end{eqnarray*}
we have
$\int_0^T \gamma(t,\widetilde{Y}_t) dt<\infty$, $\Pi^\phi_x-{\rm a.s.}$.
Note that for any $T\in (0,\infty)$, conditional on
$\sigma(\widetilde{Y})$,
$$
\sharp\left\{i: \zeta_i\in(0, T]; r_i>K\phi(\widetilde{Y}_{
\zeta_i})^{-1}{e}^{\lambda_1\zeta_i}\right\}
$$
is a Poisson random variable with intensity
$\int_0^T\gamma(t,\widetilde{Y}_t) dt$ a.s.
Hence, to prove \eqref{>K},
 we just need to prove
\begin{eqnarray}\label{infinity identity}
Z_\infty=:\int_0^\infty\gamma(t,\widetilde{Y}_t) dt=\infty,
\quad \Pi^\phi_x-{\rm a.s.}
\end{eqnarray}

Recall the choice of the constant $c>0$ in the statement above the
inequalities \eqref{domi-p}.
Applying Fubini's theorem and \eqref{domi-p}, we get
\begin{eqnarray}\label{int-ge}
\Pi^\phi_xZ_\infty
&=& \int_0^\infty dt \int_Ep^{\phi}(t, x, y)
\gamma(t, y)m(dy)\nonumber\\
&\ge&\int_c^\infty dt \int_Ep^{\phi}(t, x, y)
\gamma(t, y)m(dy)\nonumber\\
&\ge&\frac{1}{2}\int_c^\infty dt\int_E\phi(y)\widetilde{\phi}(y)\gamma(t, y)m(dy)
=:\frac{1}{2}A_\infty.
\end{eqnarray}
Exchanging the order of integration in $A_\infty$, we get that
\begin{eqnarray}\label{lowerbound}
A_\infty
&\ge&\int_E\phi(y)\widetilde{\phi}(y)\beta(y)m(dy)\sum_k
k p_k(y)I_{\{k>K\phi(y)^{-1}\}}\left[\frac{1}{\lambda_1}
\log (k\phi(y))-\frac{\log K}{\lambda_1}-c\right]^+\nonumber\\
&\ge&C_3\int_E\phi(y)\widetilde{\phi}(y)\beta(y)m(dy)\sum_k
k \log (k\phi(y))p_k(y)I_{\{k>K\phi(y)^{-1}\}}-C_4,
 \end{eqnarray}
where
$C_3=1/\lambda_1$ and $C_4=\|\beta A\|_\infty(\log K+c\lambda_1)/\lambda_1$.
The assumption that $
\int_E\widetilde{\phi}(y)\beta(y)l(y)m(dy)=\infty$ says that
$$
\int_E\phi(y)\widetilde{\phi}(y)\beta(y)m(dy)\sum_k k\log
(k\phi(y)) p_k(y)I_{\{k>\phi(y)^{-1}\}}=\infty.
$$
Since
\begin{eqnarray*}
&&\int_E\phi(y)\widetilde{\phi}(y)\beta(y)m(dy)\sum_k
k\log(k\phi(y)) p_k(y)I_{\{\phi(y)^{-1}<k\leq K\phi(y)^{-1}\}}\\
&&\leq\log K\int_E\phi(y)\widetilde{\phi}(y)\beta(y)A(y)m(dy)
<\infty,
\end{eqnarray*}
it follows from \eqref{lowerbound} that
\begin{equation}
A_\infty=\infty.
\end{equation}
Then by \eqref{int-ge},
\begin{equation}\label{E-Z-infty}
\Pi_x^\phi Z_\infty=\infty.
\end{equation}

For any finite time $T>0,$ put
$Z_T= \int_0^T\gamma(t,
\widetilde Y_t)dt$
 and
$$
\Lambda_T
=\int_c^Tdt\int_E\phi(y)\widetilde{\phi}(y)\gamma(t,y)m(dy).
$$
Then $\lim_{T\to\infty}Z_T=Z_\infty$, and $\lim_{T\rightarrow\infty}\Lambda_T=
A_\infty=\Pi^\phi_x Z_\infty=\infty$.

 An argument similar to \eqref{int-ge} yields that there exists a constant $C_5$ which is independent of $T$
and sufficiently large, such that
\begin{equation}\label{domi-mean}
\frac{1}{C_5} \Lambda_T\le\Pi_x^\phi Z_T\le C_5 \Lambda_T.
\end{equation}
By the Paley-Zygmund inequality (see, for instance, \cite[Ex.
1.3.8]{D2}),
\begin{equation}\label{Durrett-domi}
\Pi_x^\phi\left(Z_T\geq \frac{1}{2}\Pi_x^\phi Z_T\right)\geq
\frac{(\Pi_x^\phi Z_T)^2}{4\Pi_x^\phi[Z_T^2]}.
\end{equation}
 Now we estimate $\Pi_x^\phi(Z_T^2).$
\begin{eqnarray*}
 \Pi_x^\phi Z_T^2&=&\Pi_x^\phi\int_0^T\gamma(t, \widetilde Y_t)dt
 \int_0^T\gamma(s,\widetilde Y_s)ds\\
 &=&2\Pi_x^\phi\int_0^T\gamma(t, \widetilde Y_t)dt
 \int_t^T\gamma(s, \widetilde Y_s)ds\\
&=&2\Pi_x^\phi\int_0^T\gamma(t,\widetilde Y_t)dt \int_t^{(t+c)\wedge T}\gamma(s, \widetilde Y_s)ds
+ 2\Pi_x^\phi\int_0^T\gamma(t, \widetilde Y_t)dt
 \int_{(t+c)\wedge T}^{T}\gamma(s,\widetilde Y_s)ds\\
 &=:&III+IV.
\end{eqnarray*}
 By the Markov property of $\widetilde Y$, we have
 \begin{eqnarray*}IV&=&2\Pi_x^\phi\int_0^T\gamma(t, \widetilde Y_t)dt
 \Pi_{\widetilde Y_t}^\phi\int_{(c+t)\wedge T}^{T}\gamma(s,\widetilde Y_{s-t})ds\\
&=&2\Pi_x^\phi\int_0^T\gamma(t, \widetilde Y_t)dt
 \Pi_{\widetilde Y_t}^\phi\int_{(c+t)\wedge T-t}^{T-t}\gamma(u+t,\widetilde Y_{u})du\\
 &\le&2\Pi_x^\phi\int_0^T\gamma(t, \widetilde Y_t)dt
 \Pi_{\widetilde Y_t}^\phi\int_{c}^{T}\gamma(u,\widetilde Y_{u})du\\
&= & 2\Pi_x^\phi\int_0^Tdt\gamma(t, \widetilde Y_t)
\int_{c}^{T}du\int_E m(dy) p^\phi(u, \widetilde Y_t, y)\gamma(u, y),
 \end{eqnarray*}
 where in above inequality we used the fact that $\gamma(s, y)$ is decreasing in $s$,
 which is obvious by the definition of $\gamma$.
 Using \eqref{domi-p} and \eqref{domi-mean},  we see that
 $$
 IV\leq 4\Lambda_T\Pi_x^\phi Z_T \leq  4C_5(\Pi_x^\phi Z_T)^2.
 $$
Meanwhile,
\begin{eqnarray*}
III&\le &2\Pi_x^\phi\int_0^Tdt\gamma(t, \widetilde Y_t)
 \int_t^{(t+c)\wedge T}(A\beta)(\widetilde Y_s)ds\\
 &\le&2c\|\beta A\|_{\infty}\Pi_x^\phi\int_0^T\gamma(t, \widetilde Y_t)dt\\
 &=&C_6\ \Pi^{\phi}_x Z_T,
 \end{eqnarray*}
where $C_6=2c\|\beta A\|_{\infty}$.
Therefore, there exists a constant $C>0$
which does not depend on  $T$, such that
$$\Pi_x^\phi(Z_T^2)\le C(\Pi_x^\phi(Z_T))^2.$$
 Then by
\eqref{Durrett-domi} we get
$$\Pi_x^\phi\left(Z_T\geq \frac{1}{2}\Pi_x^\phi Z_T\right)\geq (4C)^{-1},$$
and therefore,
 \begin{eqnarray*}
\Pi^\phi_x\left(Z_\infty\geq \frac{1}{2}\Pi_x^\phi Z_T\right)
 \geq\Pi^\phi_x\left(Z_T\geq \frac{1}{2}\Pi_x^\phi Z_T\right)\ge
 (4C)^{-1}>0.
 \end{eqnarray*}
 Since
$\lim_{T\rightarrow\infty}\Lambda_T=\infty$, \eqref{domi-mean} and
the above inequality imply that
$\Pi^\phi_x(Z_\infty=\infty)>0$.
Since for any $t_0>0$,  $\int_0^{t_0}
\gamma(t, \widetilde Y_t)dt
\le \|\beta A\|_\infty t_0<\infty$,
the event
$\Gamma:=\left\{Z_\infty=\infty\right\}=
\left\{\int_0^\infty\gamma(t, \widetilde Y_t)dt=\infty\right\}$ is
invariant, that is, $\theta_t^{-1}(\Gamma)=\Gamma$ for all $t\ge 0$, where
$\{\theta_t: t\ge 0\}$ are the shift operators of the process $\widetilde Y$.
It follows from \eqref{IU} and \cite[Proposition X.3.9]{RY} that the
process $\widetilde Y$ is Harris-recurrent. Thus it follows from
\cite[Propositions X.3.6 and X.3.10]{RY} that
$\Pi^\phi_x\left(Z_\infty=\infty\right)=1$,
which is \eqref{infinity identity}.  And
our second conclusion follows. \qed
\end{proof}

{\bf Proof of Theorem \ref{maintheorem}.}\quad The proof heavily
depends on the decomposition \eqref{spine-decom}.

When $\int_E\widetilde{\phi}(x)\beta(x)l(x)m(dx)<\infty$, the first
conclusion of Lemma \ref{lemma1} says that
$$
\sup_{t>0}\widetilde
Q^x\left[\phi(x)M_t(\phi)\Big|\widetilde{\cG}\right] \leq
\sum_{u\in\xi}r_u\phi(\widetilde
Y_{\zeta_u})e^{-\lambda_1\zeta_u}+\|\phi\|_{\infty}<\infty.
$$
Fatou's lemma for conditional probability implies that
$\liminf_{t\to\infty}M_t(\phi)<\infty$, $\widetilde Q^x$-a.s. The
Radon-Nikodym derivative tells us that $M_t(\phi)^{-1}$ is a
nonnegative supermartingale under $Q^x$ and therefore has a finite
limit $Q^x$-a.s. So $\lim_{t\to\infty}M_t(\phi)=M_{\infty}<\infty$,
$Q^x$-a.s.
  Lemma \ref{Durrett2}
implies that in this case,
$$P^x[M_\infty(\phi)]=\lim_{t\rightarrow\infty}P^x[M_t(\phi)]=1.$$

When $\int_E\widetilde{\phi}(x)\beta(x)l(x)m(dx)=\infty$, using
the second conclusion in Lemma \ref{lemma1}, we can get under
$\widetilde Q^x,$
\[\limsup_{t\rightarrow\infty}\phi(x)M_t(\phi)
\geq\limsup_{t\rightarrow\infty}\phi(\widetilde
Y_{\zeta_{n_t}})(r_{n_t}-1)e^{-\lambda_1\zeta_{n_t}}=\infty.\]
This yields that $M_\infty(\phi)=\infty,\ Q^x$-a.s.  Using Lemma
\ref{Durrett2} again, we get $M_\infty(\phi)=0,\ P^x$-a.s. The
proof is finished.\qed

\medskip

\noindent {\bf Acknowledgment:} We thank Simon C. Harris for
sending us the paper \cite{HH1}.
We also thank two anonymous referees for their helpful comments on
the first version of this paper.

\bigskip

\begin{singlespace}

\end{singlespace}
\end{doublespace}
\vskip 0.3truein \vskip 0.3truein

 \noindent {\bf Rong-Li Liu:} LMAM School of
Mathematical Sciences, Peking
 University,  Beijing, 100871, P. R. China,
  E-mail: {\tt lrl@math.pku.edu.cn} \\

\bigskip
 \noindent {\bf Yan-Xia Ren:} LMAM School of Mathematical Sciences, Peking
 University,  Beijing, 100871, P. R. China,
  E-mail: {\tt yxren@math.pku.edu.cn} \\

\bigskip
\noindent {\bf Renming Song:} Department of Mathematics,
 The University of Illinois,  Urbana, IL 61801 U.S.A.,
  E-mail: {\tt rsong@math.uiuc.edu} \\
 \bigskip

\end{document}